\documentclass[article,onefignum,onetabnum]{siamart171218}
\usepackage{amsfonts}
\usepackage{amsmath}
\usepackage{amssymb}
\usepackage{mathrsfs}
\usepackage{graphicx}
\usepackage{placeins}
\usepackage{comment}
\usepackage{MnSymbol}
\usepackage{enumerate}
\usepackage{xfrac}
\usepackage[utf8]{inputenc}

\newcommand{\e}{\varepsilon}
\newcommand{\R}{\mathbb R}
\newcommand{\mT}{\mathbb T}

\newcommand{\beq}{\begin{equation}}
\newcommand{\eeq}{\end{equation}}

\newcommand{\dive}{\mathrm{div}\,}

\newcommand{\mres}{\hspace{-.75mm}\lefthalfcup \hspace{-.75mm}}

\newsiamremark{remark}{Remark}
\newsiamremark{hypothesis}{Hypothesis}
\crefname{hypothesis}{Hypothesis}{Hypotheses}
\newsiamthm{claim}{Claim}

\headers{A smectic liquid crystal model in the periodic setting}{Novack, Yan}

\title{A smectic liquid crystal model in the periodic setting}

\author{Michael Novack\thanks{Department of Mathematics, The University of Texas at Austin, Austin, TX, USA 
  (\email{michael.novack@austin.utexas.edu}).}
\and Xiaodong Yan\thanks{Department of Mathematics, The University of Connecticut, Storrs, CT, USA
  (\email{xiaodong.yan@uconn.edu}).}
}

\begin{document}

\maketitle

\begin{abstract}
We consider the asymptotic behavior as $\e $ goes to zero of the 2D smectics model in the periodic setting given by %
\begin{equation*}
\mathcal{E}_{\e }( w) =\frac{1}{2}\int_{\mathbb{T}^{2}}\frac{1}{%
\e }\left( \left\vert \partial _{1}\right\vert ^{-1}\left( \partial
_{2}w-\partial _{1}\frac{1}{2}w^{2}\right) \right) ^{2}+\e \left(
\partial _{1}w\right) ^{2}dx .
\end{equation*}%
We show that the energy $\mathcal{E}_\e(w)$ controls suitable $L^p$ and Besov norms of $w$ and use this to demonstrate the existence of minimizers for $\mathcal{E}_\e(w)$, which has not been proved for this smectics model before, and compactness in $L^p$ for an energy-bounded sequence. We also prove an asymptotic lower bound for $\mathcal{E}_\e(w)$ as $\e \to 0$ by means of an entropy argument.
\end{abstract}


\section{Introduction}\label{sec:intro}

We consider the variational model 
\begin{equation}
\mathcal{E}_{\e }( w) =\frac{1}{2}\int_{\mathbb{T}^{2}}\frac{1}{%
\e }\left( \left\vert \partial _{1}\right\vert ^{-1}\left( \partial
_{2}w-\partial _{1}\frac{1}{2}w^{2}\right) \right) ^{2}+\e \left(
\partial _{1}w\right) ^{2}dx\,,  \label{periodicenergy}
\end{equation}%
where $w:\mathbb{T}^{2}\rightarrow \mathbb{R}$ is a periodic function with
vanishing mean in $x_{1}$, that is 
\begin{equation}\label{vanishing mean}
\int_{0}^{1}w(x_1,x_2)\,dx_{1}=0\quad\textup{for any $x_2 \in [0,1)$}\,.
\footnotemark
\end{equation}\footnotetext{More {\color{black}generally}, a periodic distribution $f$ on $\mathbb{T}^2$ has ``vanishing mean in $x_1$'' if for all $(k_1,k_2)=k\in \left( 2\pi \mathbb{Z}\right) ^{2}$ with $k_1=0$, $\widehat{f}(k)=0$. If $f$ corresponds to an $L^p$ function, $p\in [1,\infty)$, this is equivalent to the existence of a sequence $\{\varphi_k\}$ of smooth, periodic functions satisfying  \eqref{vanishing mean} that converges in $L^p$ to $f$.}
Here $\left\vert
\partial _{1}\right\vert ^{-1}$ is defined via its Fourier coefficients
\[
\widehat{\left\vert \partial _{1}\right\vert ^{-1}f}\left( k\right)
=\left\vert k_{1}\right\vert ^{-1}\widehat{f}\left( k\right) \text{ \ for }%
k\in {\left( 2\pi \mathbb{Z}\right) ^{2}}\,,
\]
and is well defined when \eqref{vanishing mean} holds.

This  model is motivated by a nonlinear approximate model of smectic liquid
crystals. The following
functional has been proposed as an approximate model for smectic liquid
crystals \cite{BreMar99, IL99, NY1, San06, SanKam03} in two space dimensions:
\begin{equation}
E_{\e }(u)=\frac{1}{2}\int_{\Omega }\frac{1}{\e }\left( \partial_2 u-\frac{1%
}{2}(\partial_1 u)^{2}\right) ^{2}+\e (\partial_{11} u)^{2}\,dx,  \label{smecticenergy}
\end{equation}%
where $u$ is the Eulerian deviation from the ground state $\Phi(x) = x_2$ and $%
\e $ is the characteristic length scale. The first term
represents the compression energy and the second term represents the bending energy. For further background on the model, we refer to \cite{NY1,NY2} and the references contained therein. The 3D version of \eqref{smecticenergy}{\color{black}, which we analyzed in \cite{NY2} but do not consider further here,} is also used for example in the mathematical description of nuclear pasta in neutron stars \cite{CSH18}.
Assuming that $u$ is periodic on the torus $\mathbb{T}^{2}=\Omega$ and setting $w=\partial_1 u, $ (%
\ref{smecticenergy}) becomes
\[
E_{\e }(u)=\frac{1}{2}\int_{\mathbb{T}^{2}}\frac{1}{\e}\left(
|\partial _{1}|^{-1}\left( \partial _{2}w-\partial _{1}\frac{1}{2}w^{2}\right)
\right) ^{2}+\e \left( \partial _{1}w\right) ^{2}dx.
\]%
{\color{black}Finally, a similar model to \eqref{periodicenergy} with $|\partial_1|^{-1/2}$ replacing $|\partial_1|^{-1}$ has been derived in the context of micromagnetics \cite{IgnOtt19}; see also \cite{Ste11}}.

The asymptotic behavior of  \eqref{smecticenergy} as $\e$ goes to zero was studied in \cite{NY1}.  Given $\varepsilon _{n}\rightarrow 0$ and a sequence $\left\{
u_{n}\right\} $ with bounded energies $E_{\varepsilon _{n}}(
u_{n})$, the authors proved  pre-compactness of {$\{\partial_1 u_n\}$ in $L^q$ for any $1\leq q<p$ and pre-compactness of $\{\partial_2 u_n\}$  in $L^{2}$} under the additional assumption $\| \partial_1 u_{n}\| _{L^{\color{black}p }}\leq
C $ {  for some $p >6$}. 
The compactness proof in \cite{NY1} uses a compensated compactness argument based on entropies, following the work of Tartar 
\cite{Tar79, Tar83, Tar05} and Murat \cite{Mur78, Mur81ASNP, Mur81JMPA}.  In addition, a lower bound on $E_{\e}$ and a matching upper bound corresponding to a 1D ansatz was obtained as $\e \rightarrow 0$ under the assumption that the limiting function $u$ satisfies $\nabla u \in (L^\infty \cap BV)(\Omega)$.

In this paper, we approach the compactness via a different argument in the periodic setting. Our proof is motivated by recent work on related variational models in the periodic setting \cite{C-AOttSte07,GolJosOtt15, IORT20,  Ott09JFA, OttSte10,  DabJamVen20} where strong convergence of a weakly convergent $L^2$ sequence is proved via estimates on Fourier series. Given a sequence $u_{\e}$ weakly converging in $L^2(\mT^2)$, to prove strong convergence of $u_{\e}$ in $L^2$, it is sufficient to show that  there is no concentration in the high frequencies.  The center piece of this approach relies on the estimates for solutions to Burgers equation 
$$
-\partial_1\frac{1}{2}w^2+\partial_2 w=\eta
$$
in suitable  Besov spaces. This type of compactness argument also applies to a sequence $\{w_n\}$ with $\mathcal{E}_{\e}(w_n)\leq C$ for any fixed $\e$. As a direct corollary, we obtain the existence of minimizers of $E_{\e}$ in $W^{1,2}(\mT^2)$ (see Corollary  \ref{cor:exis}) for any fixed $\e$. We observe that to the best of our knowledge, the existence of minimizers of $E_{\e}$ in any setting was not known due to the lack of compactness for sequence $\{u_n\}$ satisfying $E_{\e}(u_n) \leq C$ with fixed $\e$.

To further understand the minimization of $\mathcal{E}_\e$, we are also interested in a sharp lower bound for the asymptotic limit of $\mathcal{E}_\e$ as $\e $ approaches zero. In the literature for such problems (see for example \cite{AmbDeLMan99, AviGig99,  IgnMer12, JinKoh00}), one useful technique in achieving such a bound is an ``entropy'' argument, in which the entropy production $\int \dive \Sigma(w)$ of a vector field $\Sigma(w)$ is used to bound the energy $\mathcal{E}_\e$ from below. 
For the 2D Aviles-Giga functional
\begin{equation}
\frac{1}{2}\int_\Omega \frac{1}{\e}(|\nabla u|^2 - 1)^2 + \e | \nabla^2 u|^2 \,dx\,,
\end{equation}
such vector fields were introduced in \cite{JinKoh00,DKMO01}. In \cite{NY1, NY2}, the analogue for the smectic energy, in 2D and 3D respectively, of the Jin-Kohn entropies from \cite{JinKoh00} were used to prove a sharp lower bound which can be matched by a construction similar to \cite{ ConDeL07,Pol07}. In this paper, we use the vector field
\begin{equation}\label{2dSigma}
\Sigma(w) = \left(-\frac{1}{3}w^3, \frac{1}{2}w^2 \right)
\end{equation}
which is $(-(\partial_1 u)^3/3, (\partial_1 u)^2/2)$ in terms of $u$, to prove a sharp lower bound. As $\e\to 0$, entropy production concentrates along curves and approximates the total variation of the distributional divergence of a BV vector field. An interesting open direction which motivates studying \eqref{2dSigma} is utilizing the correct version of \eqref{2dSigma} (or the entropies from \cite{DKMO01,GhiraldinLamy}) in 3D, for example in a compactness argument.


The paper is organized as follows. The pre-compactness of a sequence of functions with bounded energy is proved in Section \ref{sec:cpt}, for both fixed $\varepsilon$ and $\varepsilon \to 0$. The lower bound is established in Section \ref{sec:lbd}.

\section{Compactness of a sequence with bounded energy} \label{sec:cpt}
\subsection{Preliminaries}
\label{sec:prelim}

Let $\mathbf{e}_{1}=\left( 1,0\right) $ and $\mathbf{e}_{2}=\left( 0,1\right) $
be unit vectors in $\mathbb{R}^{2}.$ We recall some definitions from \cite%
{IORT20}. For $f:\mathbb{T}^{2}\rightarrow \mathbb{R}$, we write 
\[
\partial _{j}^{h}f\left( x\right) =f\left( x+h\mathbf{e}_{j}\right) -f\left(
x\right) \text{ \ \ \ \ \ }x\in \mathbb{T}^{2},\text{ }h\in \mathbb{R}\text{%
. }
\]

\begin{definition}
Given $f:\mathbb{T}^{2}\rightarrow \mathbb{R}$, $j\in \{1,2\},$ $s\in \left(
0,1\right] ,$ and $p\in \lbrack 1,\infty )$, the directional Besov seminorm is
defined as 
\[
\left\Vert f\right\Vert _{\overset{\cdot }{\mathcal{B}}_{p;j}^{s}}=\sup_{h%
\in \left( 0,1\right] }\frac{1}{h^{s}}\left( \int_{\mathbb{T}^{2}}\left\vert
\partial _{j}^{h}f\right\vert ^{p}dx\right) ^{\frac{1}{p}}
\]
\end{definition}

\begin{remark}
This is the $\mathcal{B}^{s;p,\infty }$ seminorm defined in each direction
separately. 
\end{remark}

\begin{remark}\label{IORT remark}
For $p=2$ and $s\in \left( 0,1\right) ,$ given $s^{\prime }\in \left(
s,1\right) ,$ the following inequality holds $\left( \cite[\textup{Equation }(2.2)]{IORT20}\right) $:%
\[
\int_{\mathbb{T}^{2}}\left\vert \left\vert \partial _{j}\right\vert
^{s}f\right\vert ^{2}=\dsum \left\vert k_{j}\right\vert ^{2s}\left\vert 
\widehat{f}\left( k\right) \right\vert ^{2}=c_{s}\int_{\mathbb{R}}\frac{1}{%
\left\vert h\right\vert ^{2s}}\int_{\mathbb{T}^{2}}\left\vert \partial
_{j}^{h}f\right\vert ^{2}dx\frac{dh}{\left\vert h\right\vert }\leq C(
s,s^{\prime }) \left\Vert f\right\Vert^{2}_{\color{black}\overset{\cdot }{\mathcal{B}}%
_{2;j}^{s'}}.
\]
\end{remark}

We quote two {\color{black}results }from \cite{IORT20}.

\begin{lemma}
\label{iortb9}\cite[Proposition B.9]{IORT20} For every $%
p\in \left( 1,\infty \right] $ and $q\in \left[ 1,p\right] $ with $\left(
p,q\right) \neq \left( \infty ,1\right) ,$ there exists a constant $C(
p,q) >0$ such that for every periodic function $f:\left[ 0,1\right)
\rightarrow \mathbb{R}$ with vanishing mean,%
\begin{equation}
\left( \int_{0}^{1}\left\vert f\left( z\right) \right\vert ^{p}dz\right) ^{%
\frac{1}{p}}\leq C( p,q) \int_{0}^{1}\frac{1}{h^{\frac{1}{q}-%
\frac{1}{p}}}\left( \int_{0}^{1}\left\vert \partial _{1}^{h}f\left( z\right)
\right\vert ^{q}dz\right) ^{\frac{1}{q}}\frac{dh}{h}\,, \label{eqn:pqbd}
\end{equation}
with the usual interpretation for $p=\infty$ or $q=\infty$.
\end{lemma}
{\color{black}The following estimate was derived in the proof of Lemma B.10 in \cite{IORT20}.
\begin{lemma}
\cite[In the proof of Lemma B.10]{IORT20} For every  $p\in \left[ 1,\infty \right) $ and
every periodic function $f:\left[ 0,1\right) \rightarrow \mathbb{R}$, $h \in (0,1]$, the following estimate holds. 
\begin{equation}
\left( \int_{0}^{1}\left\vert
\partial _{1}^{h}f\left( z\right) \right\vert ^{p}dz\right) ^{\frac{1}{p}%
}\leq 2\left( \frac{1}{h}%
\int_{0}^{h}\int_{0}^{1}\left\vert \partial _{1}^{h^{\prime }}f\left(
z\right) \right\vert ^{p}dz\,dh^{\prime }\right) ^{\frac{1}{p}}. \label{eqn:avebd}
\end{equation}
\end{lemma}}

We define  $\eta _{w}=\partial _{2}w-\partial _{1}\frac{1}{2}w^{2}$, and thus $%
\left( \ref{periodicenergy}\right) $ can be written as 
\begin{equation}\label{periodic}
\mathcal{E}_{\varepsilon}(w)=\frac{1}{2}\int_{\mT^2}\frac{1}{\varepsilon}(|\partial_1|^{-1}\eta_w)^2+\varepsilon (\partial_1w)^2 dx.
\end{equation}
Finally, we introduce the $\e$-independent energy
\begin{equation}\label{epsilon ind}
\mathcal{E}(w) = \left( \int_{\mathbb{T}^{2}}\left( \left\vert \partial
_{1}\right\vert ^{-1}\eta _{w}\right) ^{2}dx\right) ^{\frac{1}{2}}\left(
\int_{\mathbb{T}^{2}}\left( \partial _{1}w\right) ^{2}dx\right) ^{\frac{1}{2}%
}\,,
\end{equation}
and note that
\begin{equation}\label{trivial bound}
\mathcal{E}(w) \leq \mathcal{E}_\e(w)\quad\textup{
for all $\e>0$}\,.
\end{equation}
\subsection{Besov and ${\bf\textit{$L^p$}}$ estimates}
\label{sec:besov}
We obtain the following estimates. {\color{black}The proofs follow closely those in \cite[Propositions 2.3-2.4]{IORT20}.}

\begin{lemma}\label{lemma 2.6}
There exists a universal constant $C_1>0$ such that if $w\in
L^{2}\left( \mathbb{T}^{2}\right) $ and has vanishing mean in $x_{1}$ and $%
h\in \left( 0,1\color{black}\right]$, then 
\begin{equation}
\int_{\mathbb{T}^{2}}\left\vert \partial _{1}^{h}w\right\vert ^{3}dx\leq C_1h%
\mathcal{E}(w)   \label{l3estimate}
\end{equation}%
and%

\begin{equation}
{\color{black}\sup_{x_{2}\in \left[ 0,1\right) }\int_{0}^{h}\int_{0}^{1}%
\left\vert \partial _{1}^{h^{\prime }}w\left( x_{1},x_{2}\right) \right\vert
^{2}dx_{1}dh^{\prime }\leq C_1\left( h\mathcal{E}(
w) +h^{\frac{5}{3}}\mathcal{E}^{\frac{2}{3}}(
w) \right)   \label{b2sestimate}}.
\end{equation}%
\end{lemma}
\begin{proof}
{Throughout the proof, we assume that $w$ is smooth; once the estimates hold for smooth $w$, they hold in generality by approximation. The constant $C_1$ may change from line to line.} Following \cite[Equations (2.5)-(2.6)]{IORT20}, we apply the modified Howarth-K\'{a}rm\'{a}n-Monin identities for the
Burgers operator.  For every $h^{\prime }\in \left( 0,1\color{black}\right] $, we have%
\begin{equation}
\partial _{2}\frac{1}{2}\int_{0}^{1}\left\vert \partial _{1}^{h^{\prime
}}w\right\vert \partial _{1}^{h^{\prime }}w\,dx_{1}-\frac{1}{6}\partial
_{h^{\prime }}\int_{0}^{1}\left\vert \partial _{1}^{h^{\prime }}w\right\vert
^{3}dx_{1}=\int_{0}^{1}\partial _{1}^{h^{\prime }}\eta _{w}\left\vert
\partial _{1}^{h^{\prime }}w\right\vert dx_{1},  \label{HKM1}
\end{equation}%
\begin{equation}
\partial _{2}\frac{1}{2}\int_{0}^{1}\left( \partial _{1}^{h^{\prime
}}w\right) ^{2}dx_{1}-\frac{1}{6}\partial _{h^{\prime }}\int_{0}^{1}\left(
\partial _{1}^{h^{\prime }}w\right) ^{3}dx_{1}=\int_{0}^{1}\partial
_{1}^{h^{\prime }}\eta _{w}\partial _{1}^{h^{\prime }}w\,dx_{1}.  \label{HKM2}
\end{equation}
Integrating \eqref{HKM1} over $x_{2}$ and using the periodicity of $w$ yields 
\begin{eqnarray}
\partial _{h^{\prime }}\int_{\mathbb{T}^{2}}\left\vert \partial
_{1}^{h^{\prime }}w\right\vert ^{3}dx &=&-6\int_{\mathbb{T}^{2}}\partial
_{1}^{h^{\prime }}\eta _{w}\left\vert \partial _{1}^{h^{\prime
}}w\right\vert dx  \nonumber \\
&=&-6\int_{\mathbb{T}^{2}}\eta _{w}\partial _{1}^{-h^{\prime }}\left\vert
\partial _{1}^{h^{\prime }}w\right\vert dx.  \label{L3derivative}
\end{eqnarray}%
Now
\begin{eqnarray}
\left\vert \int_{\mathbb{T}^{2}}\eta _{w}\partial _{1}^{-h^{\prime
}}\left\vert \partial _{1}^{h^{\prime }}w\right\vert dx\right\vert  &\leq
&\left( \int_{\mathbb{T}^{2}}\left( \left\vert \partial _{1}\right\vert
^{-1}\eta _{w}\right) ^{2}dx\right) ^{\frac{1}{2}}\left( \int_{\mathbb{T}%
^{2}}\left( \partial _{1}\partial _{1}^{-h^{\prime }}\left\vert \partial
_{1}^{h^{\prime }}w\right\vert \right) ^{2}dx\right) ^{\frac{1}{2}} 
\nonumber \\
&\leq &C_1\left( \int_{\mathbb{T}^{2}}\left( \left\vert \partial
_{1}\right\vert ^{-1}\eta _{w}\right) ^{2}dx\right) ^{\frac{1}{2}}\left(
\int_{\mathbb{T}^{2}}\left( \partial _{1}w\right) ^{2}dx\right) ^{\frac{1}{2}%
},  \notag
\end{eqnarray}%
so that integrating \eqref{L3derivative} from $0$ to $h$ and using $\partial _{1}^{0}w=0,$
we have 
\begin{equation}
\int_{\mathbb{T}^{2}}\left\vert \partial _{1}^{h}w\right\vert ^{3}dx\leq
C_1\left( \int_{\mathbb{T}^{2}}\left( \left\vert \partial _{1}\right\vert
^{-1}\eta _{w}\right) ^{2}dx\right) ^{\frac{1}{2}}\left( \int_{\mathbb{T}%
^{2}}\left( \partial _{1}w\right) ^{2}dx\right) ^{\frac{1}{2}}h\leq C_1h%
\mathcal{E}(w) .  \notag
\end{equation}%

To prove \eqref{b2sestimate}, we integrate \eqref{HKM2} from $0$ to $h$ and again utilize  $\partial _{1}^{0}w=0$ to obtain
\begin{equation}
\partial _{2}\frac{1}{2}\int_{0}^{h}\int_{0}^{1}\left( \partial
_{1}^{h^{\prime }}w\right) ^{2}dx_{1}dh^{\prime }-\frac{1}{6}%
\int_{0}^{1}\left( \partial _{1}^{h}w\right)
^{3}dx_{1}=\int_{0}^{h}\int_{0}^{1}\partial _{1}^{h^{\prime }}\eta
_{w}\partial _{1}^{h^{\prime }}w\,dx_{1}dh^{\prime }.  \label{fderivative}
\end{equation}%
We set $$f\left( x_{2}\right) =\int_{0}^{h}\int_{0}^{1}\left( \partial
_{1}^{h^{\prime }}w\right) ^{2}dx_{1}dh^{\prime }\,,$$
and recall the Sobolev embedding inequality for $W^{1,1}\left( 
\mathbb{T}\right) \subset L^{\infty }\left( \mathbb{T}\right) $:%
\[
\sup_{z\in \mathbb{T}}\left\vert f\left( z\right) \right\vert \leq \int_{%
\mathbb{T}}\left\vert f\left( y\right) \right\vert dy+\int_{\mathbb{T}%
}\left\vert f^{\prime }\left( y\right) \right\vert dy\,.
\]%
Then applying this to $f(x_2)$ and referring to $\left( \ref%
{fderivative}\right) $, we have 
\begin{eqnarray}\label{supbd}
&&\sup_{x_{2}\in \left[ 0,1\right) }\int_{0}^{h}\int_{0}^{1}\left( \partial
_{1}^{h^{\prime }}w\right) ^{2}dx_{1}dh^{\prime } \\ \notag 
&\leq &\int_{0}^{h}\int_{%
\mathbb{T}^{2}}\left( \partial _{1}^{h^{\prime }}w\right) ^{2}dx\,dh^{\prime }
 \\ \notag
&&+\frac{1}{3}\int_{\mathbb{T}^{2}}\left\vert \partial _{1}^{h}w\right\vert
^{3}dx+2\int_{0}^{h}\int_{0}^{1}\left\vert \int_{0}^{1}\eta _{w}\partial
_{1}^{-h^{\prime }}\left\vert \partial _{1}^{h^{\prime }}w\right\vert
dx_{1}\right\vert x_{2}\,dh^{\prime }.  \nonumber
\end{eqnarray}%
Since 
\[
\int_{\mathbb{T}^{2}}\left( \partial _{1}^{h^{\prime }}w\right) ^{2}dx\leq
\left( \int_{\mathbb{T}^{2}}\left\vert \partial _{1}^{h^{\prime
}}w\right\vert ^{3}dx\right) ^{\frac{2}{3}}\leq {\color{black}C_1}\left( h^{\prime }\mathcal{E%
}( w) \right) ^{\frac{2}{3}},
\]%
and 
\begin{eqnarray*}
&&\int_{0}^{1}\left\vert \int_{0}^{1}\eta _{w}\partial _{1}^{-h^{\prime
}}\left\vert \partial _{1}^{h^{\prime }}w\right\vert dx_{1}\right\vert x_{2} \\
&\leq &\left( \int_{\mathbb{T}^{2}}\left( \left\vert \partial
_{1}\right\vert ^{-1}\eta _{w}\right) ^{2}dx\right) ^{\frac{1}{2}}\left(
\int_{\mathbb{T}^{2}}\left( \partial _{1}\partial _{1}^{-h^{\prime
}}\left\vert \partial _{1}^{h^{\prime }}w\right\vert \right) ^{2}dx\right) ^{%
\frac{1}{2}} \\
&\leq &C_1\left( \int_{\mathbb{T}^{2}}\left( \left\vert \partial
_{1}\right\vert ^{-1}\eta _{w}\right) ^{2}dx\right) ^{\frac{1}{2}}\left(
\int_{\mathbb{T}^{2}}\left( \partial _{1}w\right) ^{2}dx\right) ^{\frac{1}{2}%
},
\end{eqnarray*}%
$\left( \ref{supbd}\right) $ therefore implies  
\[
\sup_{x_{2}\in \left[ 0,1\right) }\int_{0}^{h}\int_{0}^{1}\left( \partial
_{1}^{h^{\prime }}w\right) ^{2}dx_{1}dh^{\prime }\leq C_1\left( h^{\frac{5}{3}}%
\mathcal{E}^{\frac{2}{3}}(w)+h\mathcal{E}%
( w) \right)\,, 
\]%
which is $\left( \ref{b2sestimate}\right) $. 
\end{proof}

\begin{lemma} \label{lm:lpbd}
{\color{black} If $w\in
L^{2}\left( \mathbb{T}^{2}\right) $ and has vanishing mean in $x_{1}$, then the following estimates hold:}
\begin{equation}
\left\Vert w\right\Vert _{\overset{\cdot }{\mathcal{B}}%
_{3;1}^{s}}\leq C_1\mathcal{E}^{\frac{1}{3}} (w) \,,
\text{ \ for every }s\in \left( {\color{black}0},\frac{1}{3}\right] , \label{b3sestimate}
\end{equation}
where $C_1$ is as in Lemma \ref{lemma 2.6};{\color{black}
\begin{equation}
\left\Vert w\right\Vert _{L^{p}\left( \mathbb{T}^{2}\right) }\leq
C_2(p)\mathcal{E}^{\frac{2}{3\alpha}}(w)\big(\mathcal{E}(w) + \mathcal{E}^{\frac{2}{3}}(w) \big) ^{\frac{\alpha-2}{2\alpha}}\,,  \label{lpestimate}
\end{equation}
for every $1\leq p < \frac{10}{3}$, where $\alpha= \max\{2,p \}$;} and  {\color{black}
\begin{equation}
\left\Vert w\right\Vert _{L^{p}\left( \mathbb{T}^{2}\right) }\leq C_2(p){\e^{-\frac{1}{\alpha}}}\mathcal{E}_{\e}^{\frac{1}{\alpha}}(w)\big(\mathcal{E}_\varepsilon(w) + \mathcal{E}_\varepsilon^{\frac{2}{3}}(w) \big) ^{\frac{\alpha-2}{2\alpha}} \label{eqn:lpepsilon}
\end{equation}
for every $\e >0$ and $1\leq p <6$, where again $\alpha=\max\{2,p \}$. 
}
\end{lemma}

\begin{proof}
The estimate $\left( \ref{b3sestimate}\right) $ follows from $\left( \ref{l3estimate}%
\right) $ and the definition of $\left\Vert \cdot \right\Vert _{\overset{\cdot }{%
\mathcal{B}}_{3;1}^{s}}.$ Turning to $\left( \ref{lpestimate}\right)$-{\color{black}\eqref{eqn:lpepsilon}, we first prove a preliminary estimate. We} fix $%
x_{2}\in \left[ 0,1\right)$ and apply  Lemma \ref{iortb9} to $f\left( z\right)
=w\left( z,x_{2}\right) $ {\color{black}with {\color{black}$q=2,$ $p>2$} to deduce
\[
\left( \int_{0}^{1}\left\vert w\left( x_{1},x_{2}\right) \right\vert
^{p}dx_{1}\right) ^{\frac{1}{p}}\leq C_2(p)\int_{0}^{1}\frac{1}{\color{black}h^{\frac{1}{2}-%
\frac{1}{p}}}\left( \int_{0}^{1}\left\vert \partial _{1}^{h}w\left(
x_{1},x_{2}\right) \right\vert ^{\color{black}2}dx_{1}\right) ^{\color{black}\frac{1}{2}}\frac{dh}{h}.
\]
Integrating over $x_{2},$ we thus have by Minkowski's integral inequality 

\begin{align}\notag
\left\Vert w\right\Vert _{L^{p}\left( \mathbb{T}^{2}\right) } &=\left(
\int_{0}^{1}\int_{0}^{1}\left\vert w\left( x_{1},x_{2}\right) \right\vert
^{p}dx_{1}dx_{2}\right) ^{\frac{1}{p}} \\ \notag
 &\leq C_2(p) \left(\int_0^1 \left[\int_0^1 h^{\color{black}\frac{1}{p}-\frac{3}{2}}\left(\int_0^1\left \vert\partial_1^h w(x_1,x_2)\right \vert^{\color{black}2}dx_1\right)^{\color{black}\frac{1}{2}}dh\right]^pdx_2\right)^{\frac{1}{p}}\\ \notag
&\leq C_2(p)\int_{0}^{1}h^{\color{black}\frac{1}{p}-\frac{3}{2}}\left[ \int_{0}^{1}\left( %
\int_{0}^{1}\left\vert \partial _{1}^{h}w\left( x_{1},x_{2}\right)
\right\vert ^{\color{black}2}dx_{1}\right) ^{\color{black}\frac{p}{2}}dx_{2}\right]^{\frac{1}{p}} dh \\ \notag
&\leq C_2(p)\int_{0}^{1}h^{\color{black}\frac{1}{p}-\frac{3}{2}}\sup_{x_2 \in [0,1)}\left( %
\int_{0}^{1}\left\vert \partial _{1}^{h}w\left( x_{1},x_{2}\right)
\right\vert ^{\color{black}2}dx_{1}\right) ^{\frac{p-2}{2p}}\cdot \left( \int_{\mathbb{T}^{2}}\left\vert
\partial _{1}^{h}w\left( x\right) \right\vert ^{2}dx\right) ^{\frac{1}{p}}dh\,.
\end{align}
{\color{black}The first term in the integrand can be estimated using \eqref{eqn:avebd} and \eqref{b2sestimate}, which gives
\begin{align*}
\sup_{x_2 \in [0,1)}\left( %
\int_{0}^{1}\left\vert \partial _{1}^{h}w\left( x_{1},x_{2}\right)
\right\vert ^{2}dx_{1}\right) ^{\frac{p-2}{2p}} &\leq\sup_{x_2 \in [0,1)}\left(  \frac{4}{h} \int_0^h \int_0^1 \left|\partial_1^{h'} w(x_1,x_2) \right|^2 \,dx_1dh' \right)^{\frac{p-2}{2p}} \\
&\leq C_1 \left(\mathcal{E}(w) + h^{\frac{2}{3}}\mathcal{E}^{\frac{2}{3}}(w) \right) ^{\frac{p-2}{2p}}\,,
\end{align*}
and therefore
\begin{equation}\label{startingpoint}
\| w \|_{L^p(\mathbb{T}^2)} \leq C_2(p) \big(\mathcal{E}(w) + \mathcal{E}^{\frac{2}{3}}(w) \big)^{\frac{p-2}{2p}}\int_0^1 h^{\frac{1}{p}-\frac{3}{2}}\left( \int_{\mathbb{T}^{2}}\left\vert
\partial _{1}^{h}w\left( x\right) \right\vert ^{2}dx\right)^{\frac{1}{p}}dh\,.
\end{equation}
To prove \eqref{lpestimate} and \eqref{eqn:lpepsilon} we estimate the $h$-integrand in two different fashions before integrating. For \eqref{lpestimate}, using H{\"o}lder's inequality and \eqref{l3estimate}, we have the upper bound
\begin{align}\notag
\left( \int_{\mathbb{T}^{2}}\left\vert
\partial _{1}^{h}w\left( x\right) \right\vert ^{2}dx\right) ^{\frac{1}{p}} \leq \left( \int_{\mathbb{T}^{2}}\left\vert
\partial _{1}^{h}w\left( x\right) \right\vert ^{3}dx\right) ^{\frac{2}{3p}}\leq C_1 h^{\frac{2}{3p}} \mathcal{E}^{\frac{2}{3p}}(w)\,.
\end{align}
Inserting this into \eqref{startingpoint} and using $p\in (2,10/3)$ yields
\begin{align*}
\|w\|_{L^p(\mathbb{T}^2)} &\leq C_2(p) \mathcal{E}^{\frac{2}{3p}}(w)\big(\mathcal{E}(w) + \mathcal{E}^{\frac{2}{3}}(w) \big) ^{\frac{p-2}{2p}}\int_0^1 h^{\frac{5}{3p}- \frac{3}{2}}\,dh \\ \notag
&=C_2(p)\mathcal{E}^{\frac{2}{3p}}(w)\big(\mathcal{E}(w) + \mathcal{E}^{\frac{2}{3}}(w) \big) ^{\frac{p-2}{2p}} \,,
\end{align*}
which is \eqref{lpestimate} when $p>2$. 
For $p\leq 2$, we apply \eqref{lpestimate} with $p'>2$, use the fact that $\|w\|_{L^p}\leq \| w\|_{L^{p'}}$, and let $p' \searrow 2$. Now for \eqref{eqn:lpepsilon}, we instead use the fundamental theorem of calculus and Jensen's inequality to estimate
\begin{align}\notag
\left(\int_{\mathbb{T}^2} \left|\partial_1^h w(x) \right|^2\,dx \right)^{\frac{1}{p}} &\leq \left(h^2 \int_{\mathbb{T}^2} \left(\partial_1 w(x) \right)^2\,dx \right)^{\frac{1}{p}} \\ \notag
&\leq h^{\frac{2}{p}}\varepsilon^{-\frac{1}{p}}\mathcal{E}_\varepsilon^{\frac{1}{p}}(w)\,.
\end{align}
When plugged into \eqref{startingpoint} and combined with \eqref{trivial bound}, this implies
\begin{align*}
\|w\|_{L^p(\mathbb{T}^2)} &\leq C_2(p) \varepsilon^{-\frac{1}{p}}\mathcal{E}_\varepsilon^{\frac{1}{p}}(w)\big(\mathcal{E}_\varepsilon(w) + \mathcal{E}_\varepsilon^{\frac{2}{3}}(w) \big) ^{\frac{p-2}{2p}}\int_0^1 h^{\frac{3}{p}- \frac{3}{2}}\,dh \\ \notag
&=C_2(p)\varepsilon^{-\frac{1}{p}}\mathcal{E}_\varepsilon^{\frac{1}{p}}(w)\big(\mathcal{E}_\varepsilon(w) + \mathcal{E}_\varepsilon^{\frac{2}{3}}(w) \big) ^{\frac{p-2}{2p}} 
\end{align*}
for $p\in (2,6)$. The case $p\in [1,2)$ is handled similarly as in \eqref{lpestimate}.}}
\end{proof}

\begin{remark}\color{black}
    Generalizing the previous argument to the 3D smectics model from \cite{NY2} is open. An intermediate step would be analyzing the Aviles-Giga model (which is a special case of the energy in \cite{NY2}) on $\mathbb{T}^2$ using these type of ideas.
\end{remark}

\subsection{Compactness and existence}
\label{sec:cptext}
We prove  compactness and existence theorems in this section. First we
define the admissible sets 
\[
\mathcal{A}_{\varepsilon }\mathcal{=}\left\{ w\in L^{2}\left( \mathbb{T}%
^{2}\right) :\int_{0}^{1}w\left( x_{1},x_{2}\right) dx_{1}=0\text{ for each }%
x_{2}\in \left[ 0,1\right) \text{ and }\mathcal{E}_{\varepsilon }(
w) <\infty \right\} 
\]%
{and
\[
\mathcal{A}\mathcal{=}\left\{ w\in L^{2}\left( \mathbb{T}%
^{2}\right) :\int_{0}^{1}w\left( x_{1},x_{2}\right) dx_{1}=0\text{ for each }%
x_{2}\in \left[ 0,1\right) \text{ and }\mathcal{E}(
w) <\infty \right\} .
\]
Note that for any positive $\varepsilon>0$, \eqref{trivial bound} implies that $\mathcal{A}_\varepsilon \subset \mathcal{A}$.}
We prove the following compactness result.

\begin{proposition}\label{prop:l2cpt}
If $\{w_{n}\}\subset \mathcal{A}$ satisfy $\mathcal{E}_{\varepsilon_n}( w_{n}) \leq {\color{black}C_3}<\infty$ {and $\sup_n| \varepsilon_n| \leq \varepsilon_0$},
then $\left\{ w_{n}\right\} $ is precompact in $L^{2}\left( \mathbb{T}%
^{2}\right) .$
\end{proposition}

\begin{proof}
By \eqref{lpestimate},
\begin{equation}\notag
\left\Vert w_{n}\right\Vert _{\color{black}L^2\left( \mathbb{T}%
^{2}\right) }\leq C_2(p)\mathcal{E}^{\frac{2}{3\alpha}}(w)\big(\mathcal{E}(w) + \mathcal{E}^{\frac{2}{3}}(w) \big) ^{\frac{\alpha-2}{2\alpha}} ,
\end{equation}
and {\color{black}thus, by \eqref{trivial bound} (that is, $\mathcal{E}(w)\leq \mathcal{E}_\varepsilon(w)$), $\left\Vert w_{n}\right\Vert _{L^{2}\left( 
\mathbb{T}^{2}\right) }\leq C_4$ depending on $p$ and $C_3$}. As a consequence, we can find $w_{0}\in L^{2}\left( \mathbb{T}%
^{2}\right) $ such that up to a subsequence, $w_{n}\rightharpoonup w_{0}$
weakly in $L^{2}\left( \mathbb{T}^{2}\right) .$ Therefore, for each $k\in (2\pi\mathbb{Z})^2$,
\begin{equation}
\widehat{w_{n}}\left( k\right) \rightarrow \widehat{w_{0}}\left( k\right) ,\,
\,\left\vert \widehat{w_{n}}\left( k\right) \right\vert
\leq \left( \int_{\mathbb{T}^{2}}w_{n}^{2}\right) ^{\frac{1}{2}}\leq {\color{black}C_4},\,\text{
and }\left\vert \widehat{w_{n}^{2}}\left( k\right) \right\vert \leq \int_{%
\mathbb{T}^{2}}w_{n}^{2}\leq {\color{black}C_4^{2}}.  \label{fourierbd}
\end{equation}%
We therefore know that for any fixed $N\in \mathbb{N},$ 
\[
\sum_{\substack{ \left\vert k_{1}\right\vert \leq 2\pi N, \\ \left\vert
k_{2}\right\vert \leq 2\pi N}}\left\vert \widehat{w_{n}}\left( k\right) -%
\widehat{w_{0}}\left( k\right) \right\vert ^{2}\rightarrow 0\text{ as }%
n\rightarrow \infty \,,
\]%
and so the strong convergence of $w_{n}$ $\rightarrow w_{0}$ would follow if 
\begin{equation}\label{uniform decay}
\sum_{\substack{ \left\vert k_{1}\right\vert >2\pi N \\ \textup{or} \\ \left\vert
k_{2}\right\vert >2\pi N}}\left\vert \widehat{w_{n}}\left( k\right)
\right\vert ^{2}\rightarrow 0\text{ uniformly in }n\text{ as }{N}\rightarrow
\infty .
\end{equation}
The rest of the proof is dedicated to showing \eqref{uniform decay}.

{We fix $0<s<1/3$ and appeal to Remark \ref{IORT remark} and \eqref{b3sestimate} to calculate}%
\begin{eqnarray}\notag
\int_{\mathbb{T}^{2}}\left\vert \left\vert \partial _{1}\right\vert
^{s}w_{n}\right\vert ^{2}&=&\dsum \left\vert k_{1}\right\vert ^{2s}\left\vert 
\widehat{w_{n}}\left( k\right) \right\vert ^{2}\leq C( s,\sfrac{1}{3}) \left\Vert w_{n}\right\Vert _{\overset{\cdot }{\mathcal{B}}%
_{2;1}^{1/3}}^{2}\\
&\leq &C( s,\sfrac{1}{3})\left\Vert w_{n}\right\Vert _{\overset{\cdot }{\mathcal{B%
}}_{3;1}^{1/3}}^{2}\leq C(s,\sfrac{1}{3}){\color{black}C_1}\mathcal{E}^{\frac{2}{3}%
}( w_{n}) \leq {\color{black}C_5},  \label{hsbd}
\end{eqnarray}%
{\color{black}for suitable $C_5$.} We recall the formula 
\[
\eta _{w}=\partial _{2}w-\partial _{1}\frac{1}{2}w^{2},
\]%
which, in terms of Fourier coefficients, reads 
\[
\widehat{\eta _{w}}\left( k\right) =-ik_{2}\widehat{w}\left( k\right) +\frac{%
1}{2}ik_{1}\widehat{w^{2}}\left( k\right) .
\]%
For $M_1$, $M_2\in \mathbb{N}$ to be chosen momentarily, we combine this with \eqref{fourierbd} and then \eqref{hsbd} to find %
\begin{align*}
&\sum_{\substack{ \left\vert k_{1}\right\vert >2\pi M_1 \\ \textup{or} \\\left\vert
k_{2}\right\vert >2\pi M_2}}\left\vert \widehat{w_{n}}\left( k\right)
\right\vert ^{2} \\
&\leq \sum_{\left\vert k_{1}\right\vert >2\pi M_{1}}\left\vert \widehat{%
w_{n}}\left( k\right) \right\vert ^{2}+\sum_{\substack{ \left\vert
k_{1}\right\vert \leq 2\pi M_{1} \\ \left\vert k_{2}\right\vert >2\pi M_{2}}}%
\left\vert \widehat{w_{n}}\left( k\right) \right\vert ^{2} \\
&\leq {\color{black}(2\pi M_{1})^{-2s}}\sum_{\left\vert k_{1}\right\vert >2\pi M_{1}}\left\vert
k_{1}\right\vert ^{2s}\left\vert \widehat{w_{n}}\left( k\right) \right\vert
^{2} +2\sum_{\substack{ \left\vert k_{1}\right\vert \leq 2\pi M_{1} \\ %
\left\vert k_{2}\right\vert >2\pi M_{2}}}\frac{1}{\left\vert
k_{2}\right\vert ^{2}}\left\vert \widehat{\eta _{w_{n}}}\left( k\right)
\right\vert ^{2}\\
& \qquad +{\color{black}\frac{1}{2}}\sum_{\substack{ \left\vert k_{1}\right\vert \leq 2\pi
M_{1} \\ \left\vert k_{2}\right\vert >2\pi M_{2}}}\frac{\left\vert
k_{1}\right\vert ^{2}}{\left\vert k_{2}\right\vert ^{2}}\left\vert \widehat{%
w_{n}^{2}}\left( k\right) \right\vert  \\
&\leq {\color{black}(2\pi M_{1})^{-2s}}\sum_{\left\vert k_{1}\right\vert >2\pi M_{1}}\left\vert
k_{1}\right\vert ^{2s}\left\vert \widehat{w_{n}}\left( k\right) \right\vert
^{2}+\frac{{\color{black}2}M_{1}^{2}}{M_{^{2}}^{2}}\sum_{\substack{ \left\vert
k_{1}\right\vert \leq 2\pi M_{1} \\ \left\vert k_{2}\right\vert >2\pi M_{2}}}%
\frac{1}{\left\vert k_{1}\right\vert ^{2}}\left\vert \widehat{\eta _{w_{n}}}%
\left( k\right) \right\vert ^{2}+{\color{black}\frac{C_4^2}{2}}\sum_{\substack{ \left\vert
k_{1}\right\vert \leq 2\pi M_{1} \\ \left\vert k_{2}\right\vert >2\pi M_{2}}}%
\frac{\left\vert k_{1}\right\vert ^{2}}{\left\vert k_{2}\right\vert ^{2}} \\
&\leq  {\color{black}(2\pi M_{1})^{-2s}C_5+\frac{2M_{1}^{2}}{M_{^{2}}^{2}}\times \varepsilon_0 \mathcal{E}_{\varepsilon_n}(w_n)+\frac{C_4^2}{2}\times
2(2\pi M_1)^3\times \frac{1}{\pi M_2}}.
\end{align*}%
Taking $M_{1}=M\in \mathbb{N}$ and $M_{2}=M^{4},$ we find that
\[
\sum_{\substack{ \left\vert k_{1}\right\vert >2\pi M \\ \textup{or} \\ \left\vert
k_{2}\right\vert >2\pi M^4}}\left\vert \widehat{w_{n}}\left( k\right)
\right\vert ^{2}\rightarrow 0\text{ uniformly in $n$ as }M\rightarrow \infty \,,
\]
which concludes the proof of \eqref{uniform decay}.
\end{proof}

\begin{corollary}\label{cor:lpconv}
If $\{w_{n}\}\subset \mathcal{A}$ satisfy $\mathcal{E}_{\varepsilon_n}( w_{n}) \leq C<\infty$ {and $\sup_n| \varepsilon_n| \leq \varepsilon_0$},
then $\left\{ w_{n}\right\} $ is precompact in $L^{p}\left( \mathbb{T}%
^{2}\right)$ for any $p\in [1,{\color{black}\frac{10}{3}})$.
\end{corollary}

\begin{proof}
The conclusion follows from the precompactness of ${\color{black}\{w_n\}}$ in $L^2(\mT^2)$, the bound \eqref{lpestimate} from Lemma \ref{lm:lpbd}, and interpolation. 
\end{proof}
\begin{corollary} \label{cor:fixedlp}
If $\{w_{n}\}\subset \mathcal{A}$ satisfy $\mathcal{E}_{\varepsilon}( w_{n}) \leq C<\infty$ for a fixed $\e$,
then $\left\{ w_{n}\right\} $ is precompact in $L^{p}\left( \mathbb{T}%
^{2}\right)$ for any $p\in [1,6)$.
\end{corollary}
\begin{proof}
We again appeal to the precompactness of $w_n$ in $L^2(\mT^2)$ (taking  $\e_n=\e$ in Proposition \ref{prop:l2cpt}), but instead use the bound \eqref{eqn:lpepsilon} from Lemma \ref{lm:lpbd} before interpolating. 
\end{proof}


\color{black}As a direct application of Corollary \ref{cor:fixedlp}, we can prove an existence theorem for the original smectic energy {\color{black}$E_\varepsilon$ defined in} \eqref{smecticenergy}. {\color{black}For any periodic $g:
{\mathbb{T}^1}\to \mathbb{R}$, 
we define}
\[
\widetilde{\mathcal{A}}_{\color{black}\varepsilon,g }=\left\{ u\in W^{1,2}\left( \mathbb{T}%
^{2}\right) :E_{\varepsilon }\left( u\right) <\infty,\, {\color{black}\int_0^1 u(x_1,x_2) dx_1=g(x_2)  \text{  for a.e. }x_2 \in [0,1) } \right\} .
\]
{\color{black}We note that $\widetilde{\mathcal{A}}_{\color{black}\varepsilon,g }$ is non-empty for example when $g$ is smooth. }

\begin{corollary}\label{cor:exis}
For fixed $\varepsilon >0$, {\color{black}if $\widetilde{\mathcal{A}}_{\varepsilon,g}$ is non-empty,} then there exists $u_{\varepsilon }\in \widetilde{%
\mathcal{A}}_{\color{black}\varepsilon,g}$ such that $E_{\varepsilon }\left(
u_{\varepsilon }\right) =\inf_{u\in \widetilde{\mathcal{A}}_{\color{black}\varepsilon,g
}}E_{\varepsilon }\left( u\right) .$
\end{corollary}

\begin{proof}[{\color{black}Proof}]
{\color{black}Since admissible class is non-empty, we can} let $u_{n}$ be a
minimizing sequence for
$$
E_{\varepsilon }\left( u\right)=\frac{1}{2}\int_{\Omega }\frac{1}{\e }\left( \partial_2 u-\frac{1%
}{2}(\partial_1 u)^{2}\right) ^{2}+\e (\partial_{11} u)^{2}\,dx;
$$
{\color{black}in particular, the energies are uniformly bounded. By Corollary \ref{cor:fixedlp}}, we have, up to a subsequence that we do not relabel,
\begin{equation}\label{L4 convergence}
\partial_1 u_{n}\rightarrow \partial_1 u_{0}\quad\textup{in $L^{4}\left( 
\mathbb{T}^{2}\right) $}
\end{equation} 
for some $u_{0}$.
Since $u_n$ is a minimizing sequence, the first term in $E_\varepsilon$ combined with the $L^4$-convergence of $\partial_1 u_n$ implies that $\{\partial_2 u_n \}$ are uniformly bounded in $L^2(\mathbb{T}^2)$. Thus, up to a further subsequence which we do not notate, there exists $v_0 \in L^2$ such that $\partial_2 u_n \rightharpoonup v_0$ weakly in $L^2(\mathbb{T}^2)$. Furthermore, by the uniqueness of weak limits, it must be that $v_0 = \partial_2 u_0$, so $u_0 \in W^{1,2}(\mathbb{T}^2)$. 
Expanding 
\begin{equation*}
 \int_{\Omega }\left( \partial_2 u_n-\frac{(\partial_1 u_n)^2}{2}\right)^2\,dx =\int_{\Omega }\left[ (\partial_2 u_n)^2-(\partial_1 u_n)^2\partial_2 u_n + \frac{1}{4}(\partial_1 u_n)^4\right]\,dx \,,
\end{equation*}%
we see that by \eqref{L4 convergence}, the lower semicontinuity of the $L^2$-norm under weak convergence, and the fact that  
\[
\lim_{n\rightarrow \infty }\int_{\mathbb{T}^{2}}(\partial_1 u_{n})^2 \partial_2u_{n}\,dx=\int_{%
\mathbb{T}^{2}}(\partial_1 u_{0})^2 \partial_2u_{0}\,dx,
\]%
we have
\begin{equation}\label{potential convergences}
     \liminf_{n\to \infty}\int_{\mathbb{T}^2 }\left( \partial_2 u_n-\frac{(\partial_1 u_n)^2}{2}\right)^2\,dx \geq  \int_{\mathbb{T}^2 }\left( \partial_2 u-\frac{(\partial_1 u)^2}{2}\right)^2\,dx .
\end{equation}
{\color{black}Also, the uniform $L^2$-bound on $\partial_{11} u$ and the uniqueness of limits implies that, up to a subsequence, $\partial_{11} u_n \rightharpoonup \partial_{11} u_0$ weakly in $L^2(\mathbb{T}^2)$, and thus
\begin{equation}\label{lsc of elastics}
    \liminf_{n\to \infty}\int_\Omega (\partial_{11} u_n)^2 \,dx \geq \int_{\mathbb{T}^2} (\partial_{11} u_0)^2\,dx\,.
\end{equation}
Putting together \eqref{potential convergences}-\eqref{lsc of elastics}, we} conclude 
\[
{\color{black}\inf_{\mathcal{A}_{\varepsilon,g}}E_\varepsilon =} \liminf_{n\rightarrow \infty }E_{\varepsilon }(u_{n}) \geq
E_{\varepsilon }( u_{0}) .
\]
{\color{black}Finally, {\color{black}by Poincare's inequality and the weak convergence of $\nabla u_n$ to $\nabla u_0$ in $L^2(\mathbb{T}^2)$, we conclude that $u_n$ converges  to $u_0$ strongly in $L^2(\mathbb{T}^2)$.  Hence $$\int_0^1 u(x_1,x_2) dx_1 =\lim_{n \rightarrow \infty} \int_0^1 u_n(x_1,x_2) dx_1=g(x_2) \text{ for a.e. } x_2 \in [0,1),$$ therefore}
$u_0$ belongs to $\mathcal{A}_{\varepsilon,g}$ and is a minimizer.}
\end{proof}

\section{Lower bound}\label{sec:lbd}
We consider the question of finding a limiting functional as a lower bound for $E_{\e}$ as $\e$ goes to zero. Given a sequence $\{w_\varepsilon\}$ with $E_{\e} (w_{\e})\leq C$ and $\e\to 0$, then 
\begin{equation}\label{distributionally to zero}
\int_{\mT^2}(|\partial_1|^{-1}\eta_{w_{\e}})^2 dx \rightarrow 0.
\end{equation}
Therefore $\eta_{w_{\e}} \rightarrow 0$ distributionally and the natural function  space for the limiting problem is 
$$
\mathcal{A}_0=\{w\in L^2(\mT^2):\,  \eta_w=-\partial_1 \frac{1}{2}w^2+\partial_2 w =0 \text{ in } \mathcal{D}'\}.
$$



\subsection{Properties of BV functions}
Let $\Omega \subset \mathbb{R}^2$ be a bounded open set. We first recall the BV structure theorem. For $v\in {\color{black}[BV(\Omega)]^2}$, the Radon measure $Dv$ can be decomposed as 
$$
Dv=D^a v+D^c v+D^j v 
$$
where $D^a v$ is  the absolutely continuous part of $Dv$ with respect to Lebesgue measure $\mathcal{L}^2$ and $D^c v$, $D^j v$ are the Cantor part and the jump part, respectively. All three measures are mutually singular. Furthermore, $D^a v=\nabla v \mathcal{L}^2\,\mres \Omega$ where $\nabla v$ is the approximate differential of $v$; $D^c v= D_s v \,\mres (\Omega\backslash S_v)$ and $D^j v= D_s v\, \mres{J_v} $, where $D_s v$ is the singular part of $Dv$ with respect to $\mathcal{L}^2$, $S_v$ is the set of approximate discontinuity points of $v$, and $J_v$ is the jump set of $v$. Since $J_v$ is countably $\mathcal{H}^1$-rectifiable, $D^j v$ can be expressed as 
$$
(v^+ -v^-)\otimes\nu \,\mathcal{H}^1\mres {J_v},
$$ where $\nu$ is orthogonal to the approximate tangent space at each point of $J_v$ and $v^+$, $v^-$ are the traces of $v$ from either side of $J_v$. 

Next we quote the following general chain rule formula for BV functions. 

\begin{theorem}(\cite[Theorem 3.96]{AmbFusPal00})\label{bvchain theorem}
Let $w \in {\color{black}[BV(\Omega)]^2}$, $\Omega \subset \mathbb{R}^2$, and $f\in [C^1(\R^2)]^2$  be a Lipschitz function satisfying $f(0)=0$ if $|\Omega|=\infty$. Then $v=f\circ w$ belongs to $[BV(\Omega)]^2$ and 
\begin{equation}\label{eqn:bvchain}
D v=\nabla f(w) \nabla w \mathcal{L}^2\,\mres \Omega+\nabla f(\tilde w) D^c w  + (f(w^+)-f(w^-))\otimes \nu_w \mathcal{H}^{1}\mres { J_w}.
\end{equation}
Here $\tilde{w}(x)$ is the approximate limit of $w$ at $x$ and is defined on $\Omega \backslash J_w$. 
\end{theorem}
{\color{black}In what follows, we will use Theorem 3.1 to compute the distributional divergence of such $f\circ w$ as the trace of the measure \eqref{eqn:bvchain}, that is
\begin{equation}\label{divtrace}
    \dive (f \circ w) = \mathrm{tr}\,(\nabla f(w) \nabla w) \mathcal{L}^2\,\mres \Omega + \mathrm{tr}\,(\nabla f(\tilde w) D^c w) + (f(w^+)-f(w^-))\cdot\nu_w \mathcal{H}^{1}\mres { J_w}
\end{equation}
as measures.}
{
\begin{lemma}\label{A0 lemma}
If $w\in \mathcal{A}_0 \cap (BV \cap L^\infty)(\mathbb{T}^2)$, then denoting by $D_i^a w$ and $D_i^c w$ the $i$-th components of the measures $D^a w$ and $D^c w$, we have 
\begin{equation}\notag
(-w D_1^a w + D_2^a w)  = 0\quad 
\textit{and}
\quad (-\tilde{w} D_1^c w + D_2^c w)  = 0 
\end{equation} 
{\color{black}as measures, and,} {\color{black} setting $\sigma(w) = (-w^2/2,w)$,
\begin{equation}\label{compatibility condition}
\left[ \sigma(w^+) - \sigma(w^-)\right]\cdot\nu_w = 0\quad\mathcal{H}^1\textit{-a.e. on }J_w\,.
\end{equation}}
\end{lemma}}
\begin{proof}
Let $\sigma(w) = (-w^2/2,w)$. By virtue of $w\in \mathcal{A}_0 \cap (BV \cap L^\infty)(\mathbb{T}^2)$ {\color{black}and
\eqref{divtrace},}
we know that, in the sense of distributions,
\begin{align}\notag
0 &= -\partial_1 \frac{1}{2}w^2+\partial_2 w \\ \notag
&= \dive \sigma(w) \\ \label{sum of three}
& = (-w D_1^a w + D_2^a w)  + (-\tilde{w} D_1^c w + D_2^c w)  + (\sigma(w^+) - \sigma(w^-))\cdot \nu_{w}\mathcal{H}^1\mres J_w\,.
\end{align}
But the measures $D^a w$, $D^c w$, and $D^j w$ are mutually singular, which implies that each individual term in \eqref{sum of three} is the zero measure. {\color{black}The lemma immediately follows.}
\end{proof}

\subsection{Limiting functional and the proof of the lower bound}
Let
$$
\Sigma(w)=\left(-\frac{1}{3}w^3,\frac{1}{2}w^2\right)\,.
$$
If $w \in {\color{black}\mathcal{A}_0\cap}(BV \cap  L^{\infty})(\mT^2)$, we can apply the chain rule \eqref{eqn:bvchain} and Lemma \ref{A0 lemma} to $\Sigma(w)$, yielding
\begin{align} \notag
\dive \Sigma (w) &=w(-w{\color{black}D^a_1} w+{\color{black}D^a_2}w) \mathcal{L}^2+\tilde{w}(-\tilde{w}{\color{black}D_1^c} w+{\color{black}D_2^c} w) 
\\ \notag
&\qquad+\left ( \Sigma (w^{+})-\Sigma(w^{-})\right )\cdot \nu_w \mathcal{H}^1\mres J_{w}\\ \label{eqn:dsigw}
&= \left ( \Sigma (w^{+})-\Sigma(w^{-})\right )\cdot \nu_w \mathcal{H}^1\mres J_{w}\,.
\end{align}

\begin{remark}
\color{black}Observe if  $w=u_x$  and $u_z=\frac{1}{2}u_x^2$, the entropy $\Sigma (w)$ here is exactly the entropy $\tilde{\Sigma}(\nabla u)=-(u_xu_z-\frac{1}{6}u_x^3, \frac{1}{2}u_x^2)$, which we used in the lower bound estimates in \cite{NY1}. In fact, the argument below also gives a proof of the lower bound on any domain $\Omega \subset \mathbb{R}^2$; the only necessary modification of the proof presented above is that one does not use $|\partial_1|^{-1} \eta_w$ to represent the compression energy, but rather the original expression from \eqref{smecticenergy}. 
\end{remark}

\begin{theorem}\label{thm:lower bound}
Let   $\e_n  \searrow 0$, $\{w_n\} \subset L^2(\mT^2)$ with $\partial_1 w_n \in L^2 (\mT^2)$ such that 
\begin{equation}
w_n \rightarrow w \text{ in } L^3(\mT^2),
\end{equation}
for some $w \in (BV \cap  L^{\infty})(\mT^2)$. Then 
\begin{equation}\label{liminfequation}
\liminf_{n \rightarrow \infty} \mathcal{E}_{\e_n}(w_n)\geq \int_{J_w}\frac{|w^+-w^-|^3}{12 \sqrt{1+\frac{1}{4}(w^++w^-)^2}} d\mathcal{H}^1.
\end{equation}
\end{theorem}
\begin{remark}\color{black}
    Due to recent progress on the  rectifiability for the defect set to certain solutions of Burgers equation \cite{Mar22}, the lower bound should in fact be valid among a larger class of limiting functions. Specifically, if $w\in \mathcal{A}_0 \cap L^\infty(\mathbb{T}^2)$ and for every smooth convex entropy $\Phi:\mathbb{R}\to \mathbb{R}$ and corresponding entropy flux $\Psi:\mathbb{R}\to \mathbb{R}$ with $\Psi'(v) = -\Phi'(v)v$, 
    \begin{equation}\label{finite entropy condition}
        \partial_1 \Psi(w) + \partial_2 \Phi(w) \textup{ is a finite Radon measure},
    \end{equation}
    then there exists an $H^1$-rectifiable set $J_w$ with strong traces on either side such that
    \begin{equation}\label{limiting energy}
        |\dive \Sigma(w)| = \frac{|w^+-w^-|^3}{12 \sqrt{1+\frac{1}{4}(w^++w^-)^2}} \mathcal{H}^1\mres J_w.
    \end{equation}
    In particular, by substituting any entropy/entropy flux pair for $\Sigma$ in the the argument below, one finds that for an energy bounded sequence, any limiting function $w$ satisfies \eqref{finite entropy condition} and thus \eqref{limiting energy}. Technically, applying the results of \cite{Mar22} to deduce \eqref{limiting energy} would require extending the arguments there from $[0,T] \times \mathbb{R}$ to the bounded domain $\mathbb{T}^2$ as in \cite{MarARMA,MarAdvCalcVar} and proving that \eqref{finite entropy condition} implies that $w\in C^0([0,1];L^1(\mathbb{T}^1))$ (the continuous in time dependence being a technical assumption in \cite[Definition 1.1]{Mar22}). {Regarding the regularity assumption, it is known (see e.g. \cite[Remark 5.2]{MarCalVarPDE}, \cite[pg. 191]{JOP}) that the argument of Vasseur \cite{Vas01} applies in this context and gives a representative of $w$ belonging to $C^0([0,1];L^1(\mathbb{T}^1))$. The extension of \cite{Mar22} to a bounded domain} should not present serious difficulties, although we have not pursued the details further. {\color{black}The concentration of the entropy measures on an $\mathcal{H}^1$-rectifiable jump set should be a key step in obtaining the full $\Gamma$-convergence of $E_\varepsilon$ in \eqref{smecticenergy} to the limiting energy \eqref{limiting energy}.  The remaining obstacles to such a result are the construction of a recovery sequence for functions that with gradients that do not belong to $BV \cap L^\infty$ (as the existing technology from \cite{ConDeL07,Pol07} uses both those assumptions) and the strengthening of the results of \cite{Mar22} to include functions which do not belong to $L^\infty$.}
\end{remark}
\begin{proof}[Proof of Theorem \ref{thm:lower bound}] 
Without loss of generality, we assume $\liminf_{n\rightarrow \infty} \mathcal{E}_{\e_n}(w_n) <\infty${\color{black}, so that $w\in \mathcal{A}_0$ by \eqref{distributionally to zero}. 
Now} {\color{black} for any smooth $v$}, direct calculation shows 
\begin{eqnarray}
\label{eqn:divSig}
\dive \Sigma ({\color{black}v})&=&\partial_1(-\frac{1}{3}{\color{black}v}^3)+\partial_2(\frac{1}{2}{\color{black}v}^2)\\ \notag 
&=&{\color{black}v}(\partial_2 {\color{black}v}-{\color{black}v}\partial_1{\color{black}v})={\color{black}v}\eta_{\color{black}v}. 
\end{eqnarray}
On the other hand, we can bound $\mathcal{E}_{\e}$ from below as follows:
\begin{eqnarray} \label{eqn:energybdbelow}
\mathcal{E}_{\e}({\color{black}v})&=&\frac{1}{2}\int_{\mathbb{T}^2}\frac{1}{\e}\left(|\partial_1|^{-1}\left(\partial_2 {\color{black}v}-\partial_1 \frac{1}{2} {\color{black}v}^2\right)\right)^2
+\e (\partial_1 {\color{black}v})^2 dx \\ \notag 
& = & \frac{1}{2\e} \left\Arrowvert |\partial_1|^{-1} \eta_{\color{black}v}\right\Arrowvert_{L^2(\mT^2)}^2+\frac{\e}{2}\left\Arrowvert \partial_1 {\color{black}v}\right \Arrowvert^2_{L^2(\mT^2)}\\ \notag 
& \geq & \left\Arrowvert |\partial_1|^{-1} \eta_{\color{black}v}\right\Arrowvert_{L^2(\mT^2)}\left\Arrowvert \partial_1 {\color{black}v}\right \Arrowvert_{L^2(\mT^2)}.
\end{eqnarray}
From \eqref{eqn:divSig} and \eqref{eqn:energybdbelow}, given any smooth periodic function $\phi$, for any smooth ${\color{black}v}$, we have
 \begin{eqnarray}\label{eqn:upbdsmooth}
&&\left|-\int_{{\mT}^2}\Sigma({\color{black}v})\cdot \nabla \phi\, dx \right|=\left|\int_{{\mT}^2}\dive \Sigma ({\color{black}v}) \phi \,dx\right| \\ \notag
& \leq &\left(\int_{{\mT}^2}||\partial_1|^{-1} \eta_{{\color{black}v}}|^2 dx\right)^{\frac{1}{2}}\left(\int_{{\mT}^2}|\partial_1 ({\color{black}v} \phi)|^2dx\right)^{\frac{1}{2}}\\ \notag 
& \leq &\left\Arrowvert |\partial_1|^{-1} \eta_{{\color{black}v}}\right\Arrowvert_{L^2(\mT^2)}\left\Arrowvert \partial_1 {\color{black}v}\right \Arrowvert_{L^2(\mT^2)}\left\Arrowvert \phi\right\Arrowvert_{L^{\infty}(\mT^2)}\\ \notag 
&&+\left\Arrowvert |\partial_1|^{-1} \eta_{{\color{black}v}}\right\Arrowvert_{L^2(\mT^2)}\| {\color{black}v}\|_{L^2(\mT^2)}\left\Arrowvert \partial_1\phi\right\Arrowvert_{L^{\infty}(\mT^2)}\\ \notag 
& \leq & \mathcal{E}_{\e}({\color{black}v})\left\Arrowvert \phi\right\Arrowvert_{L^{\infty}(\mT^2)}+C\sqrt{\e}\mathcal{E}_{\e}({\color{black}v})^{\frac{1}{2}}\| {\color{black}v}\|_{L^2(\mT^2)}\left\Arrowvert \partial_1\phi\right\Arrowvert_{L^{\infty}(\mT^2)} \,.
\end{eqnarray}
By the density of smooth functions in $L^2(\mT^2)$, \eqref{eqn:upbdsmooth} holds for any ${\color{black}v} \in L^2(\mT^2)$ with $|\partial_1|^{-1} \eta_{\color{black}v}, \partial_1 {\color{black}v} \in L^2(\mT^2)$. Thus 
\begin{eqnarray}
&&\left|-\int_{{\mT}^2}\Sigma(w_n)\cdot \nabla \phi \,dx \right|\\ \notag 
& \leq & \mathcal{E}_{\e_n}(w_n)\left\Arrowvert \phi\right\Arrowvert_{L^{\infty}(\mT^2)}+C\sqrt{\e_n}\mathcal{E}_{\e_n}(w_n)^{\frac{1}{2}}\|  w_n\|_{L^2(\mT^2)}\left\Arrowvert \partial_1\phi\right\Arrowvert_{L^{\infty}(\mT^2)} \,.
\end{eqnarray}
Letting $n \rightarrow \infty$, by the strong convergence of $w_n$ in $L^3(\mT^2)$, we have $\Sigma(w_n) \rightarrow \Sigma(w)$ in $L^1(\mT^2)$, so that
\begin{eqnarray}\label{eqn:limSig}
-\int_{{\mT}^2}\Sigma(w)\cdot \nabla \phi \,dx &=&-\lim_{n\rightarrow \infty}\int_{{\mT}^2}\Sigma(w_n)\cdot \nabla \phi\, dx \\ \notag
&\leq& \liminf_{n\rightarrow \infty}\mathcal{E}_{\e_n}(w_n)\left\Arrowvert \phi\right\Arrowvert_{L^{\infty}(\mT^2)}.
\end{eqnarray}
{\color{black}By taking the supremum over all smooth test functions $\phi$ with $\|\phi\|_{L^\infty}\leq 1$ in \eqref{eqn:limSig},} we see that $|\mathrm{div}\,\Sigma(w)|(\mathbb{T}^2)$ is a lower bound for the energies. To derive the explicit expression for this measure, we note that since $w\in \mathcal{A}_0 \cap (BV \cap L^\infty)(\mathbb{T}^2)$, \eqref{compatibility condition} and \eqref{eqn:dsigw} apply, so that
\begin{align*}\notag
|\mathrm{div}\,\Sigma(w)|(\mathbb{T}^2) = \left|\left[\Sigma(w^+) - \Sigma(w^-) \right]\cdot \frac{\left(\sigma(w^+)-\sigma(w^-)\right)^\perp}{|\sigma(w^+)-\sigma(w^-)|} \right| \mathcal{H}^1 \mres J_{w} \,.
\end{align*}
The right hand side of this equation can be calculated directly from the formulas for $\sigma(w)$ and $\Sigma(w)$ and simplifies to \eqref{liminfequation} (see \cite[Proof of Lemma 4.1, Equation (6.3)]{NY1}).
\end{proof}

\begin{remark}
When comparing with the lower bound proof from \cite{NY1}, this proof requires an extra integration by parts, as it does not rely on a pointwise lower bound on the energy density (see e.g. \cite[Equation (4.11)]{NY1}). The relationship between these two entropies and the structure of the corresponding arguments is exactly mirrored in the entropies devised in \cite{JinKoh00, DKMO01} for the Aviles-Giga problem - they are equal on the zero set of the potential term, and both give lower bounds, with only one of them (\cite{JinKoh00}) bounding the energy density from below pointwise.
\end{remark}


\textbf{ACKNOWLEDGEMENTS}
{\color{black}We thank both referees for useful comments that improved the article. M.N. also thanks Elio Marconi for helpful discussions regarding \cite{Mar22}.} M.N.'s research is supported by NSF grant RTG-DMS 1840314. X.Y.'s
research is supported by {\color{black} Simons Collaboration Grant \#947054, together} with a Research Excellence Grant and a CLAS Dean's summer research grant from University of
Connecticut.

\FloatBarrier 
\bibliographystyle{siam}
\bibliography{ref}
\end{document}